\documentclass[12pt]{amsart}
\usepackage{amsmath,amssymb,amsbsy,amsfonts,amsthm,latexsym,
            amsopn,amstext,amsxtra,euscript,amscd,amsthm,graphicx,comment}
\usepackage{bbm}
\usepackage[left=3.5cm,right=3.5cm,top=2cm,bottom=2cm,includeheadfoot]{geometry}
\usepackage[enableskew]{youngtab} 
\usepackage{caption}

\newtheorem{lem}{Lemma}
\newtheorem{lemma}[lem]{Lemma}

\newtheorem{prop}{Proposition}
\newtheorem{proposition}[prop]{Proposition}

\newtheorem{thm}{Theorem}
\newtheorem{theorem}[thm]{Theorem}

\newtheorem{cor}{Corollary}
\newtheorem{corollary}[cor]{Corollary}

\theoremstyle{remark}
\newtheorem*{remark}{\bf{Remarks}}

\newtheorem*{example}{{\bf Example}}

\newtheorem*{Acknowledgments}{Acknowledgments}

\theoremstyle{remar}

\def\\{\cr}
\def\({\left(}
\def\){\right)}
\def\<{\langle}
\def\>{\rangle}

\def\cC{C}
\def\cD{D}
\def\cF{\mathbb{F}}
\def\cH{\mathbb{H}}
\def\cO{\mathbb{O}}
\def\EE{\mathrm{E}}
\def\qrho{{q_{\rho}}}

\def\func#1{\mathop{\rm #1}}%

\begin{document}
\title[Asymptotic expansion]{Asymptotic expansion of Fourier coefficients of reciprocals of Eisenstein series}
\author{Bernhard Heim }
\address{Faculty of Mathematics, Computer Science, and Natural Sciences,
RWTH Aachen University, 52056 Aachen, Germany}
\email{bernhard.heim@rwth-aachen.de}
\author{Markus Neuhauser}
\address{Kutaisi International University, 5/7 Youth Avenue, Kutaisi, 4600 Georgia}
\email{markus.neuhauser@kiu.edu.ge}

\pagenumbering{arabic}
\begin{abstract}
In this paper we give a classification of the asymptotic expansion 
of the $q$-expansion of reciprocals of Eisenstein series $E_k$ of weight $k$ for the modular group
$\func{SL}_2(\mathbb{Z})$. For $k \geq 12$ even, 
this extends results of Hardy and Ramanujan, and Berndt, Bialek and Yee,
utilizing the Circle Method on the one hand, and results of Petersson, and Bringmann and Kane, 
developing a theory of meromorphic Poincar{\'e} series on the other.
We follow a uniform approach, based on the 
zeros of the Eisenstein series with the largest imaginary part. These special zeros
provide information on the singularities of the Fourier expansion of 
$1/E_k(z)$ with respect to $q = e^{2 \pi i z}$.
\end{abstract}

\maketitle
\newpage
\section{Introduction}
In this paper we provide a new approach to determine the
main asymptotic growth terms in the Fourier expansion of the reciprocals $1/E_k$ of
Eisenstein series of weight $k$. We refer to \cite{BFOR17}, Chapter 15
for a very
good introduction into the topic.

Eisenstein series are defined by
\begin{equation}
E_k(z):= 1 - \frac{2k}{B_k} \sum_{n=1}^{\infty} \sigma_{k-1}\left( n\right) \, q^n \quad (q:= e^{2 \pi i z}).
\end{equation}
They are modular forms \cite{O03} on the upper half of the complex plane $\cH $.
The algebra of modular forms with respect to the modular group $\func{SL}_2(\mathbb{Z})$ is generated by
$E_4$ and $E_6$. As usual $B_k$ denotes the $k$-th Bernoulli number and
$\sigma _{\ell }\left( n\right) := \sum_{d \mid n} d^{\ell }$.

Hardy and Ramanujan \cite{HR18B} launched in their last joint paper, 
the study of coefficients of meromorphic modular forms with a simple pole 
in the standard fundamental domain $\cF $. They demonstrated that, similar to their
famous asymptotic formula for the partition numbers 
\begin{equation}
p(n) \sim \frac{1}{4n \sqrt{3}} \, e^{ \pi \sqrt{\frac{2}{3}n}}, \qquad
\sum_{n=0}^{\infty} p(n) \, q^n := \frac{q^{\frac{1}{24}}}{\eta(z)},
\end{equation}
which had been given birth to the Circle Method \cite{HR18A},
formulas for the coefficients of
reciprocals of modular forms can be obtained.
The reciprocal of the Dedekind $\eta$-function is a weakly modular form of weight $-1/2$ on 
$\cH $.


Hardy and Ramanujan focused on the reciprocal of the Eisenstein series 
$E_6$. They proved an explicit formula
for the coefficients. Shortly afterwards, in a letter to 
Hardy, Ramanujan stated several formulas of the same type, including the $q$-expansion of $1/E_4$. No proofs
were given.

Bialek in his Ph.D.\
thesis, written under the guidance of Berndt \cite{BB05}, and finally
Berndt, Bialek and Yee \cite{BBY02} have proven the claims in the letter of Ramanujan by
extending the methods applied in \cite{HR18B}. 

We illustrate the case $k=4$. 
Following Ramanujan, we frequently put $\EE _k(q_z) := E_k(z)$ for $q=q_{z} := e^{2 \pi i z}$.
Let $\rho$
be the unique zero of $E_4$ in 
$\cF $.
Let $\lambda$ run over the integers of the form
$3^a \prod_{\ell=1}^{r} p_{\ell }^{a_{\ell }}$, where $a=0$ or $1$. Here, $p_{\ell }$ is a prime of the form $6m+1$, and $a_j \in \mathbb{N}_0$. Then \cite{BB05}:
\begin{equation}\label{E4}
\beta_4(n)= (-1)^n \frac{3}{\EE _6(q_{\rho})}   
\sum_{( \lambda )} \sum_{ (c,d)}   \frac{h_{(c,d)}(n)}{\lambda^3} \,\,
e^{\frac{\pi n \sqrt{3}}{\lambda}}.
\end{equation}
Here, $(c,d)$ runs over \emph{distinct\/}
solutions to $\lambda = c^2 - cd + d^2$, such that 
integers $a,b$ exist solving $ad-bc=1$. Let $h_{(1,0)}(n):=1$, $h_{(2,1)}(n):=(-1)^n$, and
for $\lambda \geq 7$:
\begin{equation}
h_{(c,d)}(n):=  2 \func{cos} \left( (ad+bc - 2 ac - 2bd + \lambda) \frac{ \pi \, n}{\lambda} - 6 \func{arctan}
\left( \frac{c \sqrt{3}}{2d - c}\right) \right).
\end{equation}
For the definition of distinct we refer to \cite{BB05}, Section 3.
From the explicit formula (\ref{E4}) one observes that the
main asymptotic growth comes from $(c,d)=(1,0)$. This yields 
(\cite{BK17}, Introduction):
\begin{eqnarray}
\beta_{4}(n)  & \sim & (-1)^n \frac{3}{E_6(\rho)} \, e^{\pi n \sqrt{3}}, \\
\beta_{6}(n)  & \sim & \frac{2}{E_8(i)} \, e^{2 \pi n } \label{E6},
\end{eqnarray}
where $\sum_{n=0}^{\infty} \beta_k(n) \, q^n :=
\frac{1}{E_{k}\left( z\right) }$. We added the asymptotic (\ref{E6}), which can be obtained in
a similar way.

Petersson \cite{P50} offered an alternative approach to study the $q$-expansion of
meromorphic modular forms. He defined Poincar{\'e} series with poles at arbitrary points in $\cH $ and of
arbitrary order, to provide a basis for the underlying vector spaces.
Recently, Bringmann and Kane \cite{BK17} have generalized Petersson's method. They have
also recorded several important examples.
\newline
\newline
\
In this paper we study the asymptotic expansions for all reciprocals of Eisenstein series.
Instead of proving first an explicit formula and then detecting the main growth terms, we
provide a direct approach. This is based on the distribution of the zeros
in the standard fundamental domain with the largest imaginary part.

Before we state our results, we want to point out as a warning that
the limits as $ n \rightarrow \infty$ for
$\beta_{4}\left( n\right) / \beta_{4}\left( n+1\right) $ and $\beta_{6}\left( n\right) / \beta_{6}\left( n+1\right) $
exist, but that this is not true for all $k$ as indicated in Table \ref{Table 1}.
\vspace{0.5cm}
\begin{center}
\begin{minipage}[t]{0.8\textwidth}
\begin{tabular}{|r||r|r|r|r|}
\hline
$n$ & $\frac{\beta _{4}\left( n\right) }{\beta _{4}\left( n+1\right) }\approx $ & $\frac{\beta _{6}\left( n\right) }{\beta _{6}\left( n+1\right) }\approx $ & $\frac{\beta _{12}\left( n\right) }{\beta _{12}\left( n+1\right) }\approx $ & $\frac{\beta _{14}\left( n\right) }{\beta _{14}\left( n+1\right) }\approx $ \\ \hline \hline
$1 $&$-4.3290 \cdot 10^{-3} $&$ 1.8622 \cdot 10^{-3} $&$ 5.1172 \cdot 10^{-4} $&$1.2170 \cdot 10^{-4}  $\\ \hline
$2$&$-4.3333 \cdot 10^{-3}$&$1.8677 \cdot 10^{-3}$& $-9.6536 \cdot 10^{-3}$&$4.1330 \cdot 10^{-3}   $\\ \hline
$3$&$-4.3334 \cdot 10^{-3}$&$1.8674 \cdot 10^{-3}$&$5.4260 \cdot 10^{-4}$&$1.1240 \cdot 10^{-3}$\\ \hline
$4 $&$-4.3334 \cdot 10^{-3}$&$1.8674 \cdot 10^{-3}$& $-8.9832 \cdot 10^{-3}$&$2.3564 \cdot 10^{-3}$\\ \hline
$5$&$-4.3334 \cdot 10^{-3}$&$1.8674 \cdot 10^{-3}$&$5.8359 \cdot 10^{-4}$&$1.6491 \cdot 10^{-3}$\\ \hline
$6$&$-4.3334 \cdot 10^{-3}$&$1.8674 \cdot 10^{-3}$& $-8.3936 \cdot 10^{-3}$&$1.9821 \cdot 10^{-3}$\\ \hline
$7$&$-4.3334 \cdot 10^{-3}$&$1.8674 \cdot 10^{-3}$&$6.2477 \cdot 10^{-4}$&$1.8133 \cdot 10^{-3}$\\ \hline
$\vdots $ & $\vdots $& $\vdots $& $\vdots $& $\vdots $ \\ \hline
$19$&$-4.3334 \cdot 10^{-3}$&$1.8674 \cdot 10^{-3}$&$8.8114 \cdot 10^{-4}$&$1.8674 \cdot 10^{-3}$\\ \hline
$20$&$-4.3334 \cdot 10^{-3}$&$1.8674 \cdot 10^{-3}$& $-5.6773 \cdot 10^{-3}$&$1.8674 \cdot 10^{-3}$\\ \hline
\end{tabular}
\captionsetup{margin={0cm,0cm,0cm,0cm}}
\captionof{table}{\label{Table 1}Quotients of successive coefficients of 
$1/\EE _{k}$ for $k\in \left\{ 4,6,12,14\right\} $}
\end{minipage}
\end{center}


\section{Results} \label{sect2}
The constants in the asymptotic expansion of $\beta_{k}(n)$, the coefficients of the $q$-expansion of the reciprocal of $\EE _k$, involve the Ramanujan 
$\Theta$-operator \cite{R16, BKO04} induced by residue calculation.
The differential operator 
$\Theta := q \frac{\mathrm{d}}{\mathrm{d}q}$ acts on 
formal power series by:
\begin{equation}
\Theta \left( \sum_{n=h}^{\infty} a(n) \, q^n \right) := \sum_{n=h}^{\infty} n \, a(n) \, q^n.
\end{equation}
Let $\EE _2(q):= 1 -24 \sum_{n=1}^{\infty} \sigma_1(n) \, q^n $. Ramanujan observed that
\begin{equation}
\Theta (\EE _4) 
= \left( \EE _4 \EE _2 - \EE _6   \right)/3 \text{ and } 
\Theta (\EE _6) 
= \left( \EE _6 \EE _2 - \EE _8   \right)/2.
\end{equation}
Our first results give an explicit interpretation of the
data presented in Table \ref{Table 1} for $k=6$ and $k=14$.

\begin{theorem}\label{th1}
Let $k \geq 4$ and $k \equiv 2 \pmod{4}$ be an integer. Then
$1/\EE _k$ has a $q$-expansion with radius $q_{i}=e^{-2 \pi}$: 
\begin{equation}
\frac{1}{ \EE _k(q)} = \sum_{n=0}^{\infty} \beta_k(n) \, q^n.
\end{equation}
The coefficients $\beta_k(n)$ are non-zero and have the asymptotic expansion
\begin{equation}
\beta_k(n) \sim - \frac{1}{ \Theta(\EE _k)(q_{i})}\,\,q_i^{-n}.
\end{equation}
\end{theorem}
The number $q_i= e^{-2 \pi} \approx 1.867442 \cdot 10^{-3}$ is transcendental. 
It is well-known that the so-called Gel$^{\prime }$fond constant $e^{\pi}$ is transcendental.
This was first proven by
Gel$^{\prime }$fond in 1929. It can also be deduced from the Gel$^{\prime }$fond--Schneider Theorem,
which solved Hilbert's seventh problem \cite{W08}. We refer to a result by
Nesterenko (also \cite{W08}, Section 5.6). Let $z \in \cH $. Then
at least already three of the four numbers
\begin{equation}
q_z, \EE_2(q_z), \EE_4(q_z), \text{ and } \EE_6(q_z)
\end{equation}
are algebraically independent. Since $\EE_4(q_{\rho}) = \EE_6(q_i)=0$, we obtain
that $q_i, \EE_4(q_i)$ and $q_{\rho}, \EE_6(q_{\rho})$ are transcendental.

Moreover, $\Theta(\EE _k)(q_{i})$ for $k=6,10,14$ can be explicitly expressed by
$\Gamma(\frac{1}{4})$ and $\pi$. For example,
\begin{equation}
\Theta(\EE _6)(q_{i}) = - \frac{1}{2} \EE _4(q_i)^2, \text{ where } 
\EE _4(q_i)=
\frac{3 \, \Gamma(\frac{1}{4})^8}{(2 \pi)^6}.
\end{equation}
We can also extract the numbers $q_i$ and $\EE _4(q_i)$ from the coefficients.
\begin{corollary} \label{cor1}
Let $k \geq 4$ and $k \equiv 2 \pmod{4}$. Then
\begin{eqnarray}
\lim_{n \to \infty} 
\frac{\beta_{k}\left( n\right) }{\beta_{k}\left( n+1\right) } & = & q_{i},   \label{cor1:prop1}           \\
\lim_{n \to \infty} 
\frac{\beta_{6}\left( n\right) }{\beta_{10}\left( n\right) } 
& = & \lim_{n \to \infty} 
\frac{\beta_{10}\left( n\right) }{\beta_{14}\left( n\right) } = \EE _4(q_i).  \label{extract}
\end{eqnarray}
\end{corollary}
Hardy and Ramanujan stated lower and upper bounds at the end of their initial work \cite{HR18B}
on the coefficients of the reciprocal of $1/E_6$.
We generalize their idea to all cases $k \equiv 2 \pmod{4}$ including $k=2$
and also improve their result in the original case $k=6$.
\begin{theorem}\label{th2}
Let $k \equiv 2 \pmod{4}$ and $k$ a positive integer.
Let $x_0:= \frac{2k}{B_k}$.
Then we have for all $n \in \mathbb{N}$
\begin{equation}
\frac{\left( \frac{x_{0}+\sqrt{\Delta _{k}}}{2} \right) ^{n+1}-\left( \frac{x_{0}-\sqrt{\Delta _{k}}}{2} \right) ^{n+1} }{\sqrt{\Delta _{k}}} \leq \beta_k(n)
\end{equation}
with
$\Delta _{k}=x_{0}^{2}+4\left( 2^{k-1}+1\right) x_{0}$
and
\begin{equation}
\beta_k(n) \leq \frac{\left( x_{0}-\frac{b_{k}-\sqrt{D_{k}}}{2}\right) \left( \frac{b_{k}+\sqrt{D_{k}}}{2}\right) ^{n}+\left( \frac{b_{k}+\sqrt{D_{k}}}{2}-x_{0}\right) \left( \frac{b_{k}-\sqrt{D_{k}}}{2}\right) ^{n}}{\sqrt{D_{k}}}
\label{eq:untereschranke}
\end{equation}
with
$b_{k}=x_{0}+a_{k}$,
$c_{k}=\left( 2^{k-1}+1-a_{k}\right) x_{0}$,
and $D_{k}=b_{k}^{2}+4c_{k}$ for all $k$ where $a_{2}=\sqrt{7/3}$ and
$a_{k}=\frac{3^{k-1}+1}{2^{k-1}+1}$ for $k\geq 6$.
\end{theorem}
The case $k\equiv 0 \pmod{4}$ is more complicated. For large $k$, we cannot expect that the
limit as ${n \to \infty}$ of  $\beta_k(n)/\beta_{k} (n+1)$ exists, since we have two poles on the circle of
convergence. But for $k=4$ and $k=8$ there is still only one pole. 
\begin{proposition}\label{prop1}
Let $q_{\rho} = e^{2 \pi \rho} = - e^{ - \pi \sqrt{3}}$. Let $m \in \mathbb{N}$. Then
the coefficients $\beta_{4,m}(n)$ of the $m$th power of $\EE _4^{-1}$ i.~e.\
\begin{equation}
\sum_{n=0}^{\infty}  \beta_{4,m}(n) \, q^n := \left(\frac{1}{\EE _4(q)}\right)^m 
\end{equation}
satisfy for all $m$: 
\begin{equation}
\lim_{ n \to \infty} \frac{\beta_{4,m}(n)}{\beta_{4,m}(n+1)} = q_{\rho}.
\end{equation}
\end{proposition}
\begin{remark}
\ \newline
a) For small weights the following identities exist: 
\begin{equation}\label{identities}
E_8 = E_4^2, \, E_{10}= E_4 \cdot E_6 \text{ and }E_{14} = E_4^2 \cdot E_6.
\end{equation}
b) Let the principal part of $\EE_4^{-m}$ at the pole $q_{\rho}$ be given by
\begin{equation}
\sum_{k=1}^m \frac{\lambda_{m,k}}{\left( q - q_{\rho}\right) ^{k}} \label{principal},
\end{equation}
then 
$\lambda_{m,m} = {\func{res}}_{q_{\rho}}\left( \EE_4^{-1}\right) ^m.$ 
It would be interesting to get explicit formulas for all $\lambda_{m,k}$, $1 \leq k \leq m$.
Especially for the case $m=2$.
\newline
c) We have $\func{res}_{q_{\rho}} (\EE_4^{-1}) = \frac{-3\, q_{\rho}}{\EE_6(q_{\rho})}$.
\end{remark}

We know that $\beta_4(n)$ and $\beta_8(n)$ are non-zero for all $n \in \mathbb{N}_0$
\cite{HN20B}. 
We provide new proof of the asymptotic expansion for $k=4$. This is the main term
of a formula first conjectured by Ramanujan and proven about 80 years later by
Bialek \cite{BB05}.
For the case $k=8$, we also refer to \cite{BK17}.

\begin{theorem}\label{th3}
We have $(-1)^n \beta_4(n) \in  240 \mathbb{N}$ for all $n \in \mathbb{N}$.
Further, we have the asymptotic expansion
\begin{equation}
\beta_4(n)  \sim - \frac{1}{\Theta (\EE_4) (\qrho)} \, q_{\rho}^{-n},
\end{equation}
where $\Theta (\EE_4) (\qrho) = -\EE_6(q_{\rho})/3$.
\end{theorem}

F. K. C. Rankin and H. P. F. Swinnerton-Dyer \cite{RS70} have proven that
all the zeros of $E_k(z)$ in the standard fundamental domain $\cF $
are in $\cC =\{ z \in \cF  \, : \, \vert z \vert =1 \} \subset \mathbb{F}$.
We recall the following basic facts \cite{O03}.
The modular group $\Gamma:=\func{SL}_2(\mathbb{Z})$ operates on the complex upper half plane $\cH $,
denoted by $\gamma(z)$, where $\gamma \in \Gamma $ and $z \in \cH $. 
The standard fundamental domain $\cF $ is given by
\begin{eqnarray*}\cF   &= &
\left\{ z \in \cH  \, : \, \vert z \vert \geq 1 \text{ and } 0 \leq \func{Re}\left( z\right) \leq 
1/2
\right\} \cup \\ 
& & 
\left\{ z \in \cH  \, : \, \vert z \vert > 1 \text{ and } -
1/2
< \func{Re}\left( z\right) < 0 \right\}.
\end{eqnarray*}
\begin{proposition}[Rankin, Swinnerton-Dyer \cite{RS70}]
\label{prop3}
Let $k \geq 4$ be an even integer. Let $z_k$ be the zero of $E_k$ with the largest imaginary part.
Then
\begin{equation}
z_4 = z_8 = \rho \text{ and } z_k = i \text{ for } k \equiv 2 \pmod{4}.
\end{equation}
All other zeros satisfy $z_k \in \cC  \, \backslash \, \{ i, \rho\}$.
Only for $k=8$ the zero $z_k$ is not simple.
\end{proposition}
Further, from \cite{RS70} and Kohnen \cite{K04} we obtain
\begin{corollary} \label{cor:RS}
Let $k \geq 12$ and $k \equiv 0 \pmod{4}$.
Let $k = 12 \, N + s$ for $s \in \{ 0,4,8\}$. Then
$z_k = e^{\frac{1}{2}\pi i \, \varphi}$, where $\varphi \in \left( \frac{N-1}{N},1 \right)$.
\end{corollary}
\begin{theorem}\label{th4}
Let $k$ be a positive integer. Let $k \geq 12$ and $k \equiv 0 \pmod{4}$. Then
$1/\EE_k$ has a $q$-expansion with radius $\vert q_{z_k} \vert$, where $z_k$ is the zero of $E_k$ with the
largest imaginary part. Then
\begin{equation}
\beta_k(n) \, q_{z_k}^{n} + 
\frac{1}{ \Theta(\EE_k)(q_{z_k})} + \frac{1}{ \Theta(\EE_k)(\overline{q}_{z_k})} 
\left(\frac{q_{z_k}}{\overline{q}_{z_k}}\right)^{n}
\end{equation}
constitutes a zero sequence.
\end{theorem}
The expression 
\begin{equation} \label{bound}
\frac{1}{ \Theta(\EE_k)(q_{z_k})} + \frac{1}{ \Theta(\EE_k)(\overline{q}_{z_k})} 
\left(\frac{q_{z_k}}{\overline{q}_{z_k}}\right)^{n}
\end{equation}is bounded. 
But this is not sufficient to obtain an asymptotic expansion.
Nevertheless we have discovered a new
property of the coefficients of $1/ \EE_k$ for $k \equiv 0 \pmod{4}$.
\begin{theorem} \label{th:subsequence}
Let $k \equiv 0 \pmod{4}$ and $k \geq 12$.
Then there exists a subsequence $\{n_t\}_{t=1}^{\infty}$ of $\{n\}_{n=1}^{\infty}$
such that
\begin{equation}
\lim_{t \to \infty}  \frac{\beta_k(n_t)}{
-q_{z_k}^{-n_t} 
\left( 
\frac{1}{ \Theta(\EE_k)(q_{z_k})} + \frac{1}{ \Theta(\EE_k)(\overline{q}_{z_k})} 
\left(\frac{q_{z_k}}{\overline{q}_{z_k}}\right)^{n_t}
\right)} = 1.
\end{equation}
\end{theorem}
The statement of this theorem is equivalent to
\begin{equation}
\lim_{t \to \infty} 
\frac{\beta_k(n_t) }
{- 2 \func{Re} \left( \frac{q_{z_k}^{-n_t}}{ \Theta(\EE_k)(q_{z_k})}\right)}  = 1.
\end{equation}

We further have the following properties.
\begin{theorem}\label{angle}
Let $k$ be a positive integer. Let $k \geq 12$ and $k \equiv 0 \pmod{4}$.
\begin{itemize}
\item[a)]
Let $A_k(n)$ denote the number of changes of sign in the sequence $\{\beta_k(m)\}_{m=0}^n$ 
and let $z_k = x_k + i \, y_k \in \cF $ be the zero of $E_k$ with the largest imaginary part.
Then
\begin{equation*}
\lim_{ n \to \infty} \frac{A_k(n)}{n} = 2 x_k.
\end{equation*}
\item[b)]
Let $B_k(n)$ be the number of non-zero coefficients among the $n$ coefficients
$\{\beta_k(m)\}_{m=0}^{n-1}$. Then
\begin{equation*}
\limsup _{ n \to \infty} \frac{n}{B_k(n)} \leq 2.
\end{equation*}
\end{itemize}
\end{theorem}
We end this section with a considerably surprising result.
\begin{corollary} \label{cor:0}
For large weights $k$ divisible by $4$, the coefficients
of $1/ \EE_k(q)$ satisfy
\begin{equation}
\lim_{\ell \rightarrow \infty }  \lim_{n \rightarrow \infty } \frac{A_{4\ell }(n)}{n} =0.
\end{equation}
\end{corollary}


\section{Proofs}
\subsection{Proof of Corollary \ref{cor1}, Proposition \ref{prop1}, and Theorem \ref{th1}}
We first recall a result from complex analysis.
Let $f(q)=\sum_{n=0}^{\infty} a(n) \, q^n$ be a power series regular at $q=0$ with
finite radius of convergence. Assume that there is only one singular point $q_0$ on the
circle of convergence. Let $q_0$ be a pole. Then it is known (\cite{PS78}, Part 3) that
\begin{equation}\label{lim}
\lim_{n \to \infty} \frac{a(n)}{a(n+1)} = q_0.
\end{equation}
This follows from the Laurent expansion of $f(q)$, which has a finite principal part.

Let $\EE_k(q)$ have
exactly one zero $q_0 \in B_1(0)$ with absolute value smaller than all other
zeros. Then we obtain the property (\ref{lim}) for the coefficients of $1/\EE_k$.
Note that every zero of a modular form has one representative in the fundamental domain $\cF $.

The zeros of $E_k$ are controlled by a
theorem by Rankin and Swinnerton-Dyer (\cite{RS70}, see also
Section \ref{sect2}). They proved that every zero in $\cF $ has absolute value $1$. Further, 
let $k$ be a positive, even integer and $k \geq 4$. Let $ k = 12 N +s$, where $s \in \{ 4,6,8,10,0, 14\}$.
Then $E_k$ has $N$ simple zeros in
$\cC \setminus \{i, \rho\}$. Additionally we have simple zeros 
$\rho$ for $s=4$ and $i$ for $s=6$. Further, $E_k$ has the double zero $\rho$ for $s=8$, the simple zeros $i$ and $\rho$ for $s=10$, and the simple zero $i$ and the double zero $\rho$ for $s=14$.
Further, let $z_k$ be the zero of $E_k$ with the largest imaginary part. Note that 
\begin{equation}
z_k' := J (z_k) = \left( \begin{array}{cc}
0 & -1 \\ 1 & 0 
\end{array}\right) z_k
\end{equation}
and $z_k$ have
the same imaginary part. Note that $J(i)= i$ and $J(\rho) = \rho -1$.
Thus, $1/\EE_k$ has exactly one pole on the radius of convergence iff $z_k= i$ or $z_k = \rho$.
\newline \
\begin{proof}[Proof of Corollary \ref{cor1}]
From the theorem of Rankin and Swinnerton-Dyer we obtain that for $k \equiv 2 \pmod{4}$ 
we have $z_k=i$ and $q_i = e^{-2 \pi}$.
This gives a first proof of Corollary \ref{cor1} (\ref{cor1:prop1}).
Corollary \ref{cor1} (\ref{cor1:prop1}) also follows
directly from Theorem \ref{th1}.
The quotients for small $k$ converge very quickly. 
We refer to Table \ref{Table 1} and Table \ref{more:quotients}.
\begin{center}
{\small
\begin{minipage}[t]{0.97\textwidth}
\begin{tabular}{|r||r|r|r|r|r|}
\hline
$n$ &$\frac{\beta _{8}\left( n\right) }{\beta _{8}\left( n+1\right) }\approx $& $\frac{\beta _{10}\left( n\right) }{\beta _{10}\left( n+1\right) }\approx $ & $\frac{\beta _{12}\left( n\right) }{\beta _{12}\left( n+1\right) }\approx $ & $\frac{\beta _{14}\left( n\right) }{\beta _{14}\left( n+1\right) }\approx $ & $\frac{\beta _{16}\left( n\right) }{\beta _{16}\left( n+1\right) }\approx $ \\ \hline \hline
$17$&$-4.1044\cdot 10^{-3}$&$1.8674\cdot 10^{-3}$&$8.3715\cdot 10^{-4}$&$1.8674\cdot 10^{-3}$&$1.6465\cdot 10^{-3}$\\ \hline
$18$&$-4.1159\cdot 10^{-3}$&$1.8674\cdot 10^{-3}$&$-5.9626\cdot 10^{-3}$&$1.8675\cdot 10^{-3}$&$-1.7502\cdot 10^{-2}$\\ \hline
$19$&$-4.1263\cdot 10^{-3}$&$1.8674\cdot 10^{-3}$&$8.8114\cdot 10^{-4}$&$1.8674\cdot 10^{-3}$&$2.3584\cdot 10^{-4}$\\ \hline
$20$&$-4.1357\cdot 10^{-3}$&$1.8674\cdot 10^{-3}$&$-5.6773\cdot 10^{-3}$&$1.8674\cdot 10^{-3}$&$3.8543\cdot 10^{-3}$\\ \hline
$21$&$-4.1443\cdot 10^{-3}$&$1.8674\cdot 10^{-3}$&$9.2572\cdot 10^{-4}$&$1.8674\cdot 10^{-3}$&$-1.8095\cdot 10^{-3}$\\ \hline
\end{tabular}
\captionsetup{margin={0cm,0cm,0cm,0cm}}
\captionof{table}{\label{more:quotients}Quotients of successive coefficients of 
$1/\EE _{k}$ for $k\in \left\{ 8,10,12,14,16\right\} $.}
\end{minipage}}
\end{center}
Since $\Theta( \EE_6) (q_i) = - \frac{1}{2} \EE_4(q_i)^2$ and 
$\Theta( \EE_{10}) (q_i) = - \frac{1}{2} \EE_4(q_i)^3$, the second part of the Corollary also follows
from Theorem \ref{th1} and (\ref{identities}).
An approximate numerical value of $\EE_4(q_i)$ can be read off Table \ref{E4qi}.
The theorem by Nesterenko implies that this number is transcendental, since $\EE_6(q_i)=0$.
\end{proof}
\begin{center}
{\small
\begin{minipage}[t]{0.5\textwidth}
\begin{tabular}{|r||r|}
\hline
$n$ & $\frac{\beta_{6}(n)}{\beta_{10}(n)} \approx \phantom{xxxxxxxxxx}$\\ \hline \hline
$0$ & $1.000000000000000000000000000000$ \\ \hline
$1$ & $1.909090909090909090909090909091$ \\ \hline
$2$ & $1.319410319410319410319410319410$ \\ \hline
$3$ & $1.523715744177431256188987060285$ \\ \hline
$4$ & $1.428309534304946335598514019013$ \\ \hline
$\vdots$  & $\cdots      \phantom{xxxxxxxxxxxxx}$ \\ \hline 
$80$ & $1.455762892268709322462422003594$ \\ \hline
$90$ & $1.455762892268709322462422003599$ \\ \hline
$100$ & $1.455762892268709322462422003599$ \\ \hline 
\end{tabular}
\captionsetup{margin={0cm,0cm,0cm,0cm}}
\captionof{table}{\label{E4qi}Quotients of $\beta _{6}\left( n\right) $ and $\beta _{10}\left( n\right) $.}
\end{minipage}}
\end{center}
Note that for each integer $\ell \geq 2$,  the limit as $n \to  \infty $ of 
$\frac{\beta_{4\ell -2}(n)}{\beta_{4\ell +2}(n)}$ exists, but it is  generally
not equal to $ \EE_4(q_i)$.
\begin{proof}
[Proof of Proposition \ref{prop1}]
Since $\left( 1/\EE _{4}\right) ^m$ has only the pole $q_{\rho}$ on the circle of convergence, again 
we have formula (\ref{lim}), which proves the proposition.
\end{proof}
\begin{proof}
[Proof of Theorem \ref{th1}]
Let $w$ be any complex number.
Let $B_r(w)= \{ z \in \mathbb{C} \, : \, \vert z - w \vert < r\}$ 
be the open ball with radius $r$ around $w$.
We denote the closure by $\overline{B_{r}\left( w\right) }$ and its boundary
by $\partial B_{r}\left( w\right) $.
Let $k \equiv 2 \pmod{4}$. 
Then $\EE _k$ has the special property that restricted to $\overline{B_{\vert q_i \vert }(0)}$ 
it has exactly one zero at $q_i$, which is also simple. This implies that the Taylor series expansion of the reciprocal of $\EE _k$
has  radius of convergence $\left| q_{i}\right| $ and only
a simple pole at $q_i$:
\begin{equation}
\frac{1}{ \EE _k(q)}  = \sum_{n=0}^{\infty} \beta_k(n) \, q^n  \qquad ( \vert q \vert  < \vert q_i \vert).
\end{equation}
Note that subtracting the principal part at $q_i$ provides a new Taylor series expansion with a larger radius of convergence:
\begin{equation}
\frac{1}{ \EE _k(q)} -  \frac{ \func{res}_{q_i} (1/\EE _k)}{q - q_i} = \sum_{n=0}^{\infty} b(n) \, q^n.
\end{equation}
This implies that $b(n) q_i^n$ constitutes
a zero sequence. Here, $\func{res}_{q_i} (1/\EE _k)$ denotes the residue at the pole $q_i$.
We obtain that 
\begin{equation}
q_i^{n+1} \beta_k(n) + \func{res}{}_{q_i} (1/\EE _k)
\end{equation}
constitutes a zero sequence. By a standard argument, we obtain that 
\begin{equation}
\func{res}{}_{q_i} (1/\EE _k) = \frac{1}{\frac{\mathrm{d}}{\mathrm{d}q} \EE _k (q_i)}.
\end{equation}
Finally, we obtain the asymptotic behavior
\begin{equation}
\beta_k(n) \sim - \frac{1}{\Theta (\EE _k)(q_i)}\, q_i^{-n}.
\end{equation}
\end{proof}
\subsection{Proof of Theorem \ref{th2}}
We use the following easy to prove lemmata.

\begin{lemma}
\label{sigma}$\sigma _{\ell }\left( n\right) <\frac{\ell }{\ell -1}n^{\ell }$
for $\ell >1$ and $\sigma _{1}\left( n\right) \leq \left( 1+\ln n\right) n$.
\end{lemma}

\begin{proof}
$\sigma _{\ell }\left( n\right) \leq \left( 1+\int _{1}^{n}t^{-\ell }\,\mathrm{d}t\right) n^{\ell }<\frac{\ell }{\ell -1}n^{\ell }$ for $\ell >1$ and
$\leq \left( 1+\ln n\right) n$ for $\ell =1$.
\end{proof}

\begin{lemma}
\label{abschaetzung}For $\ell \geq 5$
holds  \label{unten}$3\sqrt[\ell ]{\frac{1+3^{-\ell }}{1+2^{-\ell }}}>2.98$.
\end{lemma}

\begin{proof}
  Considering $\ell $
as a real variable $\geq 5$, we obtain the following logarithmic derivative
\begin{eqnarray*}
&&\frac{\mathrm{d}}{\mathrm{d}\ell }\frac{1}{\ell }\ln \left( \frac{1+3^{-\ell }}{1+2^{-\ell }}\right) \\
&=&-\frac{1}{\ell ^{2}}\ln \left( \frac{1+3^{-\ell }}{1+2^{-\ell }}\right) +\frac{1}{\ell }\frac{1+2^{-\ell }}{1+3^{-\ell }}\left( -\frac{3^{-\ell }\ln 3}{1+3^{-\ell }}+\frac{2^{-\ell }\ln 2}{1+2^{-\ell }}\right) >0
\end{eqnarray*}
since
$-\frac{\ln 3}{3^{\ell }+1}+\frac{\ln 2}{2^{\ell }+1}>-\frac{\ln 3}{3^{\ell }}+\frac{\ln 2}{2^{\ell +1}}>0$
for $\ell \geq 5$.
Therefore, the values of the original sequence are increasing
and we take the smallest value for $\ell =5$.
\end{proof}

\begin{proof}[Proof of Theorem \ref{th2}]
With
$\varepsilon _{k}\left( n\right) =\frac{2k}{B_{k}}\sigma _{k-1}\left( n\right) $ we
obtain
\[
E_{k}\left( z\right) =1-\sum _{n=1}^{\infty }\varepsilon _{k}\left( n\right) q^{n}.
\]
Let
$1/\left( 1-\varepsilon _{k}\left( 1\right) q-\varepsilon _{k}\left( 2\right) q^{2}\right) =\sum _{n=0}^{\infty }\alpha _{k}\left( n\right) q^{n}$.
The $\alpha _{k}\left( n\right) $
fulfill the recurrence relation
$\alpha _{k}\left( n\right) =\varepsilon _{k}\left( 1\right) \alpha _{k}\left( n-1\right) +\varepsilon _{k}\left( 2\right)  \alpha _{k}\left( n-2\right) $
for $n\geq 2$.
Obviously, $\alpha _{k}\left( 0\right) =\beta _{k}\left( 0\right) $,
$\alpha _{k}\left( 1\right) =\beta _{k}\left( 1\right) $,
and by induction
$\alpha _{k}\left( n\right) =\varepsilon _{k}\left( 1\right) \alpha _{k}\left( n-1\right) +\varepsilon _{k}\left( 2\right) \alpha _{k}\left( n-2\right) \leq \sum _{j=1}^{n}\varepsilon _{k}\left( j\right) \beta _{k}\left( n-j\right) =\beta _{k}\left( n\right) $
using the power series expansion of $1/\EE _{k}$.

For the upper bound let
$a_{2}=\sqrt{7/3}$ and for $k\geq 6$ let
$a_{k}=\frac{\varepsilon _{k}\left( 3\right) }{\varepsilon _{k}\left( 2\right) }=\frac{\sigma _{k-1}\left( 3\right) }{\sigma _{k-1}\left( 2\right) }=\frac{3^{k-1}+1}{2^{k-1}+1}$.
For all $k\equiv 2\pmod{4}$ let
$b_{k}=a_{k}+\varepsilon _{k}\left( 1\right) $,
$c_{k}=\varepsilon _{k}\left( 2\right) -a_{k}\varepsilon _{k}\left( 1\right) $,
and
$\frac{1-b_{k}q-c_{k}q^{2}}{1-a_{k}q}=1-\sum _{n=1}^{\infty }\delta _{k}\left( n\right) q^{n}$.
Therefore, $\delta _{k}\left( 1\right) =b_{k}-a_{k}=\varepsilon _{k}\left( 1\right) $,
$\delta _{k}\left( 2\right) =c_{k}+a_{k}\delta _{k}\left( 1\right) =\varepsilon _{k}\left( 2\right) $,
and
$\delta _{k}\left( n\right) =a_{k}\delta _{k}\left( n-1\right) $ for $n\geq 3$. Therefore
$\delta _{k}\left( n\right) =\varepsilon _{k}\left( 2\right) a_{k}^{n-2}$.

\begin{enumerate}
\item  First, let $k=2$. Then
$\delta _{2}\left( n\right) =72\left( 7/3\right) ^{\left( n-2\right) /2}$.
For $n\in \left\{ 3,4,5,6\right\} $ we obtain
$24\sigma _{1}\left( n\right) \leq \delta _{2}\left( n\right) $. Using
Lemma~\ref{sigma} we obtain $\varepsilon _{2}\left( n\right) \leq 24\left( 1+\ln n\right) n$.
For $n=7$ we obtain
$24\cdot \left( 1+\ln 7\right) \cdot 7<504<72\left( 7/3\right) ^{\left( 7-2\right) /2}$
and for $n\geq 7$ we obtain
$\frac{1+\ln \left( n+1\right) }{1+\ln n}\frac{n+1}{n}\leq \left( 1+\frac{\ln \left( 1+\frac{1}{7}\right) }{1+\ln n}\right) \frac{8}{7}<1.2<\sqrt{7/3}$.
Therefore, $\varepsilon _{2}\left( n\right) \leq \delta _{2}\left( n\right) $.

\item  Now, let $k\geq 6$
then
\[
\delta _{k}\left( n\right) =\varepsilon _{k}\left( 3\right) a_{k}^{n-3}=\frac{2k}{B_{k}}\left( 3^{k-1}+1\right) \left( \left( \frac{3}{2}\right) ^{k-1}\frac{1+3^{1-k}}{1+2^{1-k}}\right) ^{n-3}.
\]
Using Lemma~\ref{sigma} we obtain
$\sigma _{k-1}\left( n\right) <\frac{k-1}{k-2}n^{k-1}$. Since $k\geq 6$
by Bernoulli's
inequality $\frac{k-1}{k-2}\leq \frac{5}{4}=1+\frac{1}{4}<\left( 1+\frac{1}{20}\right) ^{5}
\leq \left( \frac{21}{20}\right) ^{k-1}$.
Therefore
\[
\sqrt[k-1]{\frac{B_{k}}{2k}\varepsilon _{k}\left( n\right) }=\sqrt[k-1]{\sigma _{k-1}\left( n\right) }<\sqrt[k-1]{\frac{k-1}{k-2}n^{k-1}}<\frac{21}{20}n.
\]
Using Lemma~\ref{abschaetzung} implies
$\sqrt[k-1]{\frac{B_{k}}{2k}\delta _{k}\left( n\right) }>2.98\left( \frac{3}{2}\right) ^{n-3}$.
Now $\frac{21}{20}n<2.98\left( \frac{3}{2}\right) ^{n-3}$ for $n\geq 4$ as
$4.2<4.47$ for $n=4$ and $\frac{n}{n-1}<\frac{3}{2}$
for $n>4$.
\end{enumerate}

We have shown 
$\varepsilon _{k}\left( n\right) =\delta _{k}\left( n\right) $
for $n\in \left\{ 1,2\right\} $ and
$\varepsilon _{k}\left( n\right) \leq \delta _{k}\left( n\right) $ for
all $n\geq 3$.
Let now
$\frac{1-a_{k}q}{1-b_{k}q-c_{k}q^{2}}=\sum _{n=0}^{\infty }\gamma _{k}\left( n\right) q^{n}$.
Then $\beta _{k}\left( n\right) =\gamma _{k}\left( n\right) $
for $n\in \left\{ 1,2\right\} $ and by induction
$\gamma _{k}\left( n\right) =\sum _{j=1}^{n}\delta _{k}\left( j\right) \gamma _{k}\left( n-j\right) \geq \sum _{j=1}^{n}\varepsilon _{k}\left( j\right) \beta _{k}\left( n-j\right) =\beta _{k}\left( n\right) $
for $n\geq 3$.

We have shown 
$\alpha _{k}\left( n\right) \leq \beta _{k}\left( n\right) \leq \gamma _{k}\left( n\right) $
for all $n\geq 1$. From the generating functions we can now
determine formulas
for $\alpha _{k}\left( n\right) $ and $\gamma _{k}\left( n\right) $.
The characteristic equation for $\alpha _{k}\left( n\right) $ is
$\lambda _{k}^{2}-\varepsilon _{k}\left( 1\right) \lambda _{k}-\varepsilon _{k}\left( 2\right) =0$.
Let
$\Delta _{k}=\varepsilon _{k}\left( 1\right) ^{2}+4\varepsilon _{k}\left( 2\right) =\left( \frac{2k}{B_{k}}\right) ^{2}+\frac{8k}{B_{k}}\left( 2^{k-1}+1\right) $.
Then
$\lambda _{k,\pm }=\frac{1}{2}\left( \varepsilon _{k}\left( 1\right) \pm \sqrt{\Delta _{k}}\right) $.
We obtain
$\left(
\begin{array}{c}
L_{k,+} \\
L_{k,-}
\end{array}
\right) =\left(
\begin{array}{cc}
1 & 1 \\
\lambda _{k,+} & \lambda _{k,-}
\end{array}
\right) ^{-1}\left(
\begin{array}{c}
1 \\
\varepsilon _{k}\left( 1\right) 
\end{array}
\right) =\frac{1}{\lambda _{k,+}-\lambda _{k,-}}\left(
\begin{array}{c}
\varepsilon _{k}\left( 1\right)  -\lambda _{k,-} \\
\lambda _{k,+}-\varepsilon _{k}\left( 1\right) 
\end{array}
\right) =\frac{1}{\sqrt{\Delta _{k}}}\left(
\begin{array}{c}
\lambda _{k,+} \\
-\lambda _{k,-}
\end{array}
\right) $.
Therefore,
$\alpha _{k}\left( n\right) =L_{k,+}\lambda _{k,+}^{n}+L_{k,-}\lambda _{k,-}^{n}=\frac{\lambda _{k,+}^{n+1}-\lambda _{k,-}^{n+1}}{\sqrt{\Delta _{k}}}$
for all $n$.

The characteristic equation for $\gamma _{k}\left( n\right) $ is
$\mu _{k}^{2}-b_{k}\mu _{k}-c_{k}=0$.
Let
$D _{k}=b_{k}^{2}+4c_{k}$.
Then
$\mu _{k,\pm }=\frac{1}{2}\left( b_{k}\pm \sqrt{D_{k} }\right) $,
\[
\left(
\begin{array}{c}
M_{k,+} \\
M_{k,-}
\end{array}
\right) =\left(
\begin{array}{cc}
1 & 1 \\
\mu _{k,+} & \mu _{k,-}
\end{array}
\right) ^{-1}\left(
\begin{array}{c}
1 \\
\varepsilon _{k}\left( 1\right) 
\end{array}
\right) =\frac{1}{\sqrt{D_{k} }}
\left(
\begin{array}{c}
\varepsilon _{k}\left( 1\right) -\mu _{k,-} \\
\mu _{k,+}-\varepsilon _{k}\left( 1\right) 
\end{array}
\right) ,
\]
and
$\gamma _{k}\left( n\right) =M_{k,+}\mu _{k,+}^{n}+M_{k,-}\mu _{k,-}^{n}$.
\end{proof}

\begin{example}[Slight improvement of \cite{HR18B}]
Let $k=6$. Then
\begin{eqnarray*}
\alpha _{6}\left( n\right) &=&\frac{1}{\sqrt{320544}}\left( \left( \frac{504+\sqrt{320544}}{2}\right) ^{n+1}-\left( \frac{504-\sqrt{320544}}{2}\right) ^{n+1}\right) \\
&\approx &\frac{1}{566.16}\left( 535.08^{n+1}-\left( -31.083\right) ^{n+1}\right) .
\end{eqnarray*}
With $x_{0}=\frac{12}{B_{6}}=504$,
$a_{6}=\frac{244}{33}$,
$b_{6}=\frac{16876}{33}$,
$c_{6}=\frac{141960}{11}$,
$D_{6}=\frac{341015536}{1089}$
and $\sqrt{D_{6}}\approx 559.59$ we obtain
$\mu _{6,\pm }
=
\frac{b_{6}\pm \sqrt{D_{6}}}{2}
$,
\begin{equation*}
M_{6,+}
=
\frac{1}{\sqrt{D_{6}}}\left( x_{0}-\frac{b_{6}-\sqrt{D_{6}}}{2}\right)
,\qquad
M_{6,-}=\frac{1}{\sqrt{D_{6}}}\left( \frac{b_{6}+\sqrt{D_{6}}}{2}-x_{0}\right)
.
\end{equation*}
By (\ref{eq:untereschranke})
this finally yields
\begin{equation*}
\gamma_{6}\left( n\right) =M_{6,+}\mu _{6,+}^{n}+M_{6,-}\mu _{6,-}^{n}
\approx \frac{528.10\cdot 535.49^{n}+31.494\cdot \left( -24.100\right) ^{n}}{559.59}
.
\end{equation*}
The second and last column in Table \ref{HRbounds} are the
lower and upper bounds from \cite{HR18B}. 
\begin{table}
\[
\begin{array}{|c||c|c|c|c|c|}
\hline
n&\frac{535 ^{n+1}-\left( -31\right) ^{n+1}}{566} &\alpha _{6}\left( n\right) &\beta _{6}\left( n\right) &\gamma _{6}\left( n\right) &\frac{352 \cdot 535.5^{n}+21\left( -24\right) ^{n}}{373} \\ \hline \hline
1&5.0400\cdot 10^{2}&5.0400\cdot 10^{2}&5.0400\cdot 10^{2}&5.0400\cdot 10^{2}&5.0400\cdot 10^{2}\\ \hline
2&2.7060\cdot 10^{5}&2.7065\cdot 10^{5}&2.7065\cdot 10^{5}&2.7065\cdot 10^{5}&2.7065\cdot 10^{5}\\ \hline
3&1.4474\cdot 10^{8}&1.4479\cdot 10^{8}&1.4491\cdot 10^{8}&1.4491\cdot 10^{8}&1.4491\cdot 10^{8}\\ \hline
4&7.7438\cdot 10^{10}&7.7475\cdot 10^{10}&7.7600\cdot 10^{10}&7.7600\cdot 10^{10}&7.7602\cdot 10^{10}\\ \hline
5&4.1429\cdot 10^{13}&4.1456\cdot 10^{13}&4.1554\cdot 10^{13}&4.1554\cdot 10^{13}&4.1556\cdot 10^{13}\\ \hline
6&2.2165\cdot 10^{16}&2.2182\cdot 10^{16}&2.2252\cdot 10^{16}&2.2252\cdot 10^{16}&2.2253\cdot 10^{16}\\ \hline
7&1.1858\cdot 10^{19}&1.1869\cdot 10^{19}&1.1916\cdot 10^{19}&1.1916\cdot 10^{19}&1.1917\cdot 10^{19}\\ \hline
8&6.3441\cdot 10^{21}&6.3511\cdot 10^{21}&6.3807\cdot 10^{21}&6.3809\cdot 10^{21}&6.3813\cdot 10^{21}\\ \hline
9&3.3941\cdot 10^{24}&3.3983\cdot 10^{24}&3.4168\cdot 10^{24}&3.4169\cdot 10^{24}&3.4172\cdot 10^{24}\\ \hline
\end{array}
\]
\caption{\label{HRbounds} Improvement of upper and lower bounds (approximation) for $\beta_6(n)$.}
\end{table}
\end{example}
\subsection{Proof of Theorem \ref{th3}}
For the special case of $k=4$ we refer to a result of
\cite{HN20B}. We have proven 
that $(-1)^n \beta_4(n) \in 240 \, \mathbb{N}$ for all $n \in \mathbb{N}$
(see also \cite{AKN97}, last section, for an announcement of the result of
strict sign changes).
We are mainly interested in the implication $\beta_4(n) \neq 0$.  
\newline
\begin{proof}[Proof of Theorem \ref{th3}]
Let $k=4$. Then $z_4= \rho$ and $J(z_4) = \rho -1$. This implies
that $1/\EE_4(q) = \sum_{n=0}^{\infty} \beta_4(n) \, q^n$ has $\vert q_{\rho} \vert$ as the
radius of convergence. Further, the only singularity on the circle of convergence is given by the pole
$q_{\rho}$. Now we can proceed as in the proof of Theorem \ref{th1} and obtain the asymptotic expansion of $\beta_4(n)$. Here we use the
fact that $ \func{res}_{q_{\rho }} \EE_4^{-1}$ is equal to 
\begin{equation}
\frac{q_{\rho}}{\Theta\left( \EE _{4}\right) \left( q_{\rho}\right) }= \frac{ -3 \, q_{\rho}}{\EE_6(q_{\rho})}.
\end{equation}
\end{proof}

\subsection{Proof of Theorem \ref{th4} and Theorem \ref{th:subsequence}}
\begin{proof}
[Proof of Theorem \ref{th4}]
Let $k \equiv 0 \pmod{4}$. We are interested in the zeros
of $E_k$ which contribute to poles on the circle of convergence of 
the power series
\begin{equation}
\frac{1}{\EE_k(q)} = \sum_{n=0}^{\infty} \beta_k(n) \, q^n.
\end{equation}
Let $k \geq 12$ then Proposition \ref{prop3} and Corollary \ref{cor:RS} imply
that there are exactly two singularities provided by the two poles at
$q_{z_k}$ and $\overline{q}_{z_k}$. This implies that the radius of convergence
is equal to $\left| q_{z_k} \right| $.
Here we also used the well-known fact,
that the imaginary part of $\gamma(z)$, when $\gamma$ is in the
modular group and $z$ in the fundamental domain, does not increase.
Next we consider the Laurent expansion of $1 / \EE_k(q)$ around $q_{z_k}$.
We subtract the principal part from $1  /  \EE_k(q)$ and obtain a holomorphic
function at $q_{z_k}$. We iterate this procedure and consider the Laurent expansion 
around the other pole $\overline{q}_{z_k}$ and subtract again the principal part.
Note that we have poles of order one.
This implies that
\begin{equation} \label{subtract}
\frac{1}{\EE _{k}\left( q\right) }
- 
\frac{ \func{res}_{q_{z_k}} \EE_k^{-1}}{q - q_{z_k}} - 
\frac{ \func{res}_{\overline{q}_{z_k}} \, \, \EE_k^{-1}}{q - \overline{q}_{z_k}}
\end{equation}
has a holomorphic expansion $\sum_{n=0}^{\infty} b(n) \, q^n$, with a radius of
convergence larger than $\vert q_{z_k}\vert  = \left| \overline{q_{z_k}}\right| $.
This implies that $b(n) q_{z_k}^n$ and $b(n) \overline{q}_{z_k}^{n}$
constitute zero sequences.
The residue values can be expressed by $\Theta(\EE_k)$ evaluated at the poles.
This leads to an expression which allows in the final formula
the number $q_{z_k}^{-n}$ to appear instead of $q_{z_k}^{-(n+1)}$.
See also the proof of Theorem \ref{th1}.
By the identity principle $b(n)$ is equal to
\begin{equation}
\beta_{k}(n) +  \frac{1}{\Theta \left(\EE_k\right)  (q_{z_k})} \, q_{z_k}^{-n} + 
\frac{1}{\Theta \left(\EE_k\right) 
(\overline{q}_{z_k})} \, \overline{q}_{z_k}^{-n}. 
\end{equation}

This implies that 
\begin{equation*}
\sum_{n=0}^{\infty} \left( \beta_{k}(n) +  \frac{1}{\Theta \left(\EE_k\right)  (q_{z_k})} \, q_{z_k}^{-n} + 
\frac{1}{\Theta \left(\EE_k\right) 
(\overline{q}_{z_k})} \, \overline{q}_{z_k}^{-n} \right) q^{n} = \sum_{n=0}^{\infty}
b(n) q_{z_k}^n \, \left( \frac{q}{{q}_{z_k}}\right)^n 
\end{equation*}
for $ q \in \mathbb{C}$ and $ \left| q\right| < \vert q_{z_k} \vert$.
Let $w= q/ q_{z_k}$.
Then 
\begin{equation*}
\sum_{n=0}^{\infty} \left( \beta_{k}(n)  q_{z_k}^{n} +  \frac{1}{\Theta \left(\EE_k\right)  (q_{z_k})} + 
\frac{1}{\Theta \left(\EE_k\right) 
(\overline{q}_{z_k})} \, 
\left( \frac{q_{z_k}}{
\overline{q}_{z_k}}\right)^n
\right) w^n  = \sum_{n=0}^{\infty}
b(n) q_{z_k}^n \, w^n .
\end{equation*}
In the final step we compare the coefficients with respect to
$w^{n}$ and use the
identity principle for regular power series. Since $b(n) \, q_{z_k}^n$
constitutes a zero sequence, the 
claim of the theorem follows.
\end{proof}

\begin{proof}[Proof of Theorem \ref{th:subsequence}]
Let $k \equiv 0 \pmod{4}$ and $k \geq 12$. Let $z_k = x_k +iy_k$
be the zero of $E_k$ in $\cF $ with the largest imaginary part. Then
$z_k \neq i, \rho$. This implies by results
by Kanou \cite{K00} and Kohnen \cite{K03} that
$z_k$ is transcendental. Since we have chosen $z_k$ on the circle of unity, we
can conclude that $x_k$ and $y_k$ are also transcendental. 
By a well-known result by
Kronecker \cite{K84}, since $x_k$ is irrational, the orbit
\begin{equation}
\cO _k := \left\{  \left(  \frac{ q_{z_k}}{\overline{q}_{z_k}}\right)^n \, : n \in \mathbb{N} \right\}
\end{equation}
is dense in $\left\{ w= e^{2 \pi i \alpha} \, : \, \alpha \in [0,1) \right\} $.
Let $C_k:= 1/ \Theta(\EE_k)(q_{z_k})$. Since 
\begin{equation}
\overline{C}_k= 1/ \Theta(\EE_k)(\overline{q}_{z_k}),
\end{equation} for the closure of the set, we obtain
\begin{equation}
\cD _k:= \left\{
\frac{1}{ \Theta(\EE_k)(q_{z_k})} + \frac{1}{ \Theta(\EE_k)(\overline{q}_{z_k})} 
\left(\frac{q_{z_k}}{\overline{q}_{z_k}}\right)^{n}\, : \, n \in \mathbb{N} \right\}
\end{equation}
a circle with center $C_k$ and radius $\vert C_k\vert $:
\begin{equation}
\partial B_{\vert C_k \vert} (C_k)= \Big{\{} z \in \mathbb{C} \, : \, \vert z - C_k \vert = \vert C_k \vert \Big{\}}.
\end{equation}
We note that $0$ and $2 \, C_k$ are not elements of $\cD _k$.
Let $d_k \in \partial B_{\vert C_k \vert} (C_k)\setminus \left\{ 0\right\} $.
Then there exists a
subsequence $\left\{ n_{t}\right\} _{t=1}^{\infty }$
of $\{n\}_{n=1}^{\infty}$ such that,
\begin{equation}
\lim_{ t \to \infty} \frac{1}{ \Theta(\EE_k)(q_{z_k})} + \frac{1}{ \Theta(\EE_k)(\overline{q}_{z_k})} 
\left(\frac{q_{z_k}}{\overline{q}_{z_k}}\right)^{n_t} = d_k
\end{equation}
Combining this result with Theorem \ref{th4} proves the claim.
\end{proof}


\subsection{Proof of Theorem \ref{angle} and 
Corollary \ref{cor:0}}
We recall a result from complex analysis. 
P{\'o}lya and Szeg{\H{o}} recorded the following beautiful property
(\cite{PS78}, Part Three, Chapter~5).
Let $f(x)= \sum_{n=0}^{\infty} a(n) \, x^n$ be a power series with radius of convergence $ 0 < r < \infty$
and real coefficients.
We assume that we have only two singularities on the circle of convergence and that these two
singularities are poles:  $x_1 = r e^{i \alpha}$ and $x_2 = r e^{- i \alpha}$ with $0 < \alpha < \pi$.
Let $A(n)$ denote the number of changes of sign in the sequence $\{a(m)\}_{m=0}^n$. Then $\lim_{n \to \infty} \frac{A(n)}{n} = \frac{\alpha}{\pi}$. The number of changes of sign in a sequence of real numbers is given by the sign changes of the sequence, when all zeros are removed. Results in this direction had
also been given by K{\"o}nig \cite{K75} in 1875.
\begin{proof}[Proof of Theorem \ref{angle}, part a)]
Let $k \equiv 0 \pmod{4}$. Then $1/\EE_k(q)= \sum_{n=0}^{\infty} \beta_k(n) \, q^n$ has
a radius of convergence $\vert q_{z_k} \vert$, where $z_k = x_k + i y_k$ is the
zero of $E_k$ with the largest imaginary part with $0 < x_k < 1/2$.
We stated already that 
$q_{z_k}$ and $\overline{q}_{z_k}$ are the single two singularities on the circle of convergence.
Note that $q_{z_k} = r_k \cdot e^{2 \pi i x_k}$, where 
$r_k = e^{-2\pi y_k} = \vert q_{z_k} \vert$.
Further, $\overline{q}_{z_k} = r_k \cdot e^{-2 \pi i x_k}$. Thus all assumptions
are fulfilled to apply the above cited result for $A(n)= A_k(n)$ and $\alpha = 2 x_k$.
\end{proof}
\begin{example} 
We have
$z_{16}\approx 0.196527+ 0.980498 \, i$. See Table \ref{signs} for
values $A_{16}\left( n\right) /n$.
\begin{table}
\[
\begin{array}{|r||r|r|r|}
\hline
n&\frac{A_{16}\left( n\right) }{n}\approx &\frac{A_{16}\left( 10n\right) }{10n}\approx &\frac{A_{16}\left( 100n\right) }{100n}\approx \\ \hline \hline
2&0.50000000&0.40000000&0.39500000\\\hline
3&0.33333333&0.40000000&0.39333333\\\hline
4&0.50000000&0.40000000&0.39250000\\\hline
5&0.40000000&0.40000000&0.39400000\\\hline
6&0.33333333&0.40000000&0.39333333\\\hline
7&0.42857143&0.40000000&0.39285714\\\hline
8&0.37500000&0.40000000&0.39375000\\\hline
9&0.44444444&0.38888889&0.39333333\\\hline
10&0.40000000&0.39000000&0.39300000\\\hline
\end{array}
\]
\caption{\label{signs}Portion of sign changes for $k=16$.}
\end{table}

\end{example}
We also recall another interesting result stated in \cite{PS78} (Part Three,
Chapter 5).
Let $f(x)= \sum_{n=0}^{\infty} a(n) \, x^n$ be a power series with finite positive radius of 
convergence. We assume that there are
only poles on the circle of convergence.
Let $B(n)$ be the number of non-zero coefficients among the first $n$ coefficients $\{a(m)\}_{m=0}^{n-1}$.
Then the number of poles is not smaller than
\begin{equation}\label{non-zero}
\limsup _{ n \to \infty} \frac{n}{B(n)}.
\end{equation}
\begin{proof}[Proof of Theorem \ref{angle}, part b)]
The number of poles is $2$. Thus, by the result above, two is an upper bound for the term
(\ref{non-zero}), which completes the proof.
\end{proof}
\begin{example} We have $B_{12}(n) = B_{16}(n) = B_{20}(n)=n$ for $n \leq 1000$.
\end{example}
\begin{proof}[Proof of Corollary \ref{cor:0}]
From Theorem \ref{angle} we obtain
\begin{equation}
\lim_{n \to \infty} 
\frac{A_{4 \ell}}{4\ell} = 2 \, x_{4 \ell},
\end{equation}
where $x_{4 \ell}$ is the real part of the zero of $E_{4 \ell}$ with the largest imaginary part.
Finally, from Corollary \ref{cor:RS} the claim follows, since $x_{4 \ell}$ tends to zero.
\end{proof}

\begin{Acknowledgments}
To be entered later.
\end{Acknowledgments}


\end{document}